\documentclass[intlimits,12pt,reqno]{amsart}
\usepackage{latexsym,amsthm,amsmath,amssymb,mathrsfs,mathtools}
\usepackage[T1]{fontenc} 
\usepackage[utf8]{inputenc} 
\usepackage[mathscr]{eucal}

\usepackage{tikz}
\usetikzlibrary{arrows,calc,patterns,cd}
\usepackage{wrapfig}

\def\mvint_#1{\mathchoice
          {\mathop{\vrule width 6pt height 3 pt depth -2.5pt
                  \kern -9pt \intop}\limits_{\kern -3pt #1}}%
          {\mathop{\vrule width 5pt height 3 pt depth -2.6pt
                  \kern -6pt \intop}\nolimits_{#1}}%
          {\mathop{\vrule width 5pt height 3 pt depth -2.6pt
                  \kern -6pt \intop}\nolimits_{#1}}%
          {\mathop{\vrule width 5pt height 3 pt depth -2.6pt
                  \kern -6pt \intop}\nolimits_{#1}}}

\newcommand{\bbbr}{\mathbb R}
\newcommand{\bbbb}{\mathbb B}
\newcommand{\bbbd}{\mathbb D}

\newcommand{\bbbz}{\mathbb Z}
\newcommand{\bbbi}{\mathbb I}

\newcommand{\R}{\mathbb R}
\newcommand{\overbar}[1]{\mkern 1.7mu\overline{\mkern-1.7mu#1\mkern-1.5mu}\mkern 1.5mu}

\newcommand{\eps}{\varepsilon}

\def\rank{\operatorname{rank}}

\newtheorem{theorem}{Theorem}
\newtheorem*{theorem*}{Theorem}
\newtheorem{lemma}[theorem]{Lemma}

\newtheorem{conjecture}[theorem]{Conjecture}
\theoremstyle{definition}
\newtheorem{remark}[theorem]{Remark}
\newtheorem*{remark*}{Remark}
\newtheorem{example}{Example}

\newcommand{\Sph}{\mathbb S}

\usepackage{setspace}

\title[$C^1$ mappings with derivative of small rank]{$C^1$ mappings in $\bbbr^5$ with derivative
of rank at most $3$ cannot be
uniformly approximated by $C^2$ mappings with derivative of rank at most 3}

\author[P. Goldstein]{Pawe\l{}  Goldstein}
\address{Pawe\l{} Goldstein, Institute of Mathematics, Faculty of Mathematics, Informatics and Mechanics, University of Warsaw, Banacha 2, 02-097 Warsaw, Poland} \email{P.Goldstein@mimuw.edu.pl}
\thanks{P.G. was partially supported by National Science Center grant no 2012/05/E/ST1/03232.}
\author[P. Haj\l{}asz]{Piotr Haj\l{}asz}
\address{Piotr Haj\l{}asz, Department of Mathematics, University of Pittsburgh, Pittsburgh, PA 15260, USA}
\email{hajlasz@pitt.edu}
\thanks{P.H.\ was supported by NSF grant DMS-1800457}

\begin{document}
\sloppy


\begin{abstract}
We find a counterexample to a conjecture of Gałęski \cite{Galeski2017} by constructing for some positive integers $m<n$
a mapping $f\in C^1(\bbbr^n,\bbbr^n)$
satisfying $\rank Df\leq m$ that, even locally, cannot be uniformly approximated by $C^2$ mappings $f_\eps$ satisfying
the same rank constraint: $\rank Df_\eps\leq m$.
\end{abstract}

\maketitle

\section{Introduction}
In the context of geometric measure theory
Jacek Gałęski \cite[Conjecture~1.1 and ~Section~3.3]{Galeski2017} formulated the following conjecture.
\begin{conjecture}
\label{con1}
Let $1\leq m<n$ be integers and let $\Omega\subset\bbbr^n$ be open.
If $f\in C^1(\Omega,\bbbr^n)$ satisfies $\rank Df\leq m$ everywhere in $\Omega$, then $f$ can be uniformly approximated by smooth mappings $g\in C^{\infty}(\Omega,\bbbr^n)$ such that $\rank Dg\leq m$ everywhere in $\Omega$.
\end{conjecture}

A weaker form of the conjecture is whether any mapping as in  Conjecture~\ref{con1} can be approximated locally.
\begin{conjecture}
\label{con2}
Let $1\leq m<n$ be integers and let $\Omega\subset\bbbr^n$ be open.
If $f\in C^1(\Omega,\bbbr^n)$ satisfies $\rank Df\leq m$ everywhere in $\Omega$, then for every point $x\in\Omega$ there is 
a neighborhood $\bbbb^n(x,\eps)\subset\Omega$ and a sequence $f_i\in C^\infty(\bbbb^n(x,\eps),\bbbr^n)$ such that $\rank Df_i\leq m$ and $f_i$ converges to $f$ 
uniformly on $\bbbb^n(x,\eps)$.
\end{conjecture}

The following result is easy to prove and it shows that Conjecture~\ref{con2} is true on an open and dense subset of $\Omega$.
\begin{theorem}
\label{thm:easy}
Let $1\leq m<n$ be integers and let $\Omega\subset\bbbr^n$ be open.
If $f\in C^1(\Omega,\bbbr^n)$ satisfies $\rank Df\leq m$ everywhere in $\Omega$, then there is an open and dense set $G\subset\Omega$ such that
for every point $x\in G$ there is 
a neighborhood $\bbbb^n(x,\eps)\subset G$ and a sequence $f_i\in C^\infty(\bbbb^n(x,\eps),\bbbr^n)$ such that $\rank Df_i\leq m$ and $f_i$ converges to $f$ 
uniformly on $\bbbb^n(x,\eps)$.
\end{theorem}

However, in general Conjecture~\ref{con2} (and hence Conjecture~\ref{con1}) is false and the main result of the paper
provides a family counterexamples for certain ranges of 
$n$ and $m$.  

\begin{theorem}
\label{thm:main}
Suppose that $m+1\leq k<2m-1$, $\ell\geq k+1$, $r\geq m+1$, and the homotopy group $\pi_k(\Sph^m)$ is non-trivial.
Then there is a map $f\in C^1(\bbbr^\ell, \bbbr^r)$ with $\rank Df\leq m $ in $\bbbr^\ell$ and a Cantor set $E\subset \bbbr^\ell$ with the following property:
\begin{quote}
For every $x_o\in E$ and $\eps>0$ there is $\delta>0$ such that\\
if ${g\in C^{k-m+1}(\bbbb^\ell(x_o,\eps),\bbbr^r)}$ and
$|f(x)-g(x)|<\delta \text{ for all } x\in\bbbb^\ell(x_o,\eps)$,\\
then $\rank Dg\geq m+1$ on a non-empty open set in $\bbbb^\ell(x_o,\eps)$.
\end{quote}
\end{theorem}

(Here by a Cantor set we mean a set that is homeomorphic to the ternary Cantor set.)

Therefore the mapping $f$ cannot be approximated in the supremum norm by $C^{k-m+1}$ mappings with rank of the derivative $\leq m$
in any neighborhood of any point of the set $E$.

\begin{remark}
In fact, the mapping $f$  constructed in the proof of Theorem~\ref{thm:main} is $C^\infty$ smooth on $\bbbr^\ell\setminus E$, so $G=\bbbr^\ell\setminus E$ is an open and dense 
set where we can approximate $f$ smoothly, cf.\ Theorem~\ref{thm:easy}.
\end{remark}

Since the assumptions of the theorem are quite complicated, let us show explicit situations when the approximation cannot hold.

\begin{example}
\label{jeden}
If $n\geq 3$, $\ell\geq n+2$ and $r\geq n+1$, then there is $f\in C^1(\bbbr^\ell,\bbbr^r)$ with $\rank Df\leq n$ in $\bbbr^\ell$ that cannot
be locally approximated in the supremum norm by mappings $g\in C^2(\bbbr^\ell,\bbbr^r)$ satisfying $\rank Dg\leq n$.

Indeed, if $n\geq 3$, $k=n+1$ and $m=n$, then $\pi_k(\Sph^m)=\bbbz_2$ (see \cite{hatcher}) and $m+1\leq k<2m-1$.

In particular, there is $f\in C^1(\bbbr^5,\bbbr^5)$ with $\rank Df\leq 3$ that cannot be locally approximated in the supremum norm by mappings
$g\in C^2(\bbbr^5,\bbbr^5)$ satisfying $\rank Dg\leq 3$.
\end{example}
\begin{example}
$\pi_6(\Sph^4)=\bbbz_2$, $k=6$, $m=4$, $m+1\leq k<2m-1$, so there is $f\in C^1(\bbbr^7,\bbbr^7)$, $\rank Df\leq 4$, that cannot be locally approximated by mappings
$g\in C^3(\bbbr^7,\bbbr^7)$ satisfying $\rank Dg\leq 4$.
\end{example}
\begin{example}
$\pi_8(\Sph^5)=\bbbz_{24}$, $k=8$, $m=5$, $m+1\leq k<2m-1$, so there is $f\in C^1(\bbbr^9,\bbbr^9)$, $\rank Df\leq 5$, that cannot be locally approximated by
mappings $g\in C^4(\bbbr^9,\bbbr^9)$ satisfying $\rank Dg\leq 5$.
\end{example}

Infinitely many essentially different situations when the assumptions of Theorem~\ref{thm:main} are satisfied can be easily obtained by examining the catalogue of homotopy groups of spheres.

While, in general, Gałęski's conjecture is not true, Theorem~\ref{thm:main} covers only a certain range of dimensions and ranks, leaving other cases unsolved. We believe that
the following special case of the conjecture is true.

\begin{conjecture}
If $f\in C^1(\bbbr^n,\bbbr^k)$, $n,k\geq 2$, satisfies $\rank Df\leq 1$, then $f$ can be uniformly
approximated (at least locally) by mappings $g\in C^\infty(\bbbr^n,\bbbr^k)$ satisfying $\rank Dg\leq 1$.
\end{conjecture}

Our belief is based on the fact that in that case the structure of the mapping $f$ is particularly simple: on the open set where
$\rank Df=1$, it is a $C^1$ curve that branches on the set where $\rank Df=0$.

\section{Proof of Theorem~\ref{thm:easy}}
Let $G\subset\Omega$ be the set of points where the function $x\mapsto \rank Df(x)$ attains a local maximum i.e.,
$$
G=\{x\in\Omega:\, \exists \eps>0\ \forall y\in\bbbb^n(x,\eps)\  \rank Df (y)\leq \rank Df(x)\}.
$$
We claim that the set $G$ is open, and that $\rank Df$ is locally constant in $G$. 
Indeed, the set $\{\rank Df\geq k\}$ is open so if $x\in G$ and $\rank Df(x)=k$, then $\rank Df(y)\geq k$ in a neighborhood $\bbbb^n(x,\eps)$ of $x$, but $\rank Df$ attains a local
maximum at $x$, so $\rank Df(y)=k$ in $\bbbb^n(x,\eps)$. Clearly, $\bbbb^n(x,\eps)\subset G$ and
$\rank Df$ is constant in the neighborhood $\bbbb^n(x,\eps)\subset G$.

We also claim that the set $G\subset\Omega$ is dense. Let $x\in \Omega$ and $\bbbb^n(x,\eps)\subset\Omega$. Since $\rank Df$ can attain only a finite number of values,
it attains a local maximum at some point $y\in \bbbb^n(x,\eps)$. Clearly, $y\in G$. That proves density of $G$.

It remains to prove now that $f$ can be locally approximated in $G$. Let $x\in G$. Then $\rank Df(x)=k\leq m$. Since $\rank Df$ is constant in a neighborhood of $x$, it follows from
the Rank Theorem \cite[Theorem~8.6.2/2]{zorich} that there are diffeomorphisms $\Phi$ and $\Psi$ defined in neighborhoods of $x$ and $f(x)$ respectively such that
$\Phi(x)=0$, $\Psi(f(x))=0$, and
$$
\Psi\circ f\circ\Phi^{-1}(x_1,\ldots,x_n)=(x_1,\ldots,x_k,0,\ldots,0)
\quad
\text{in a neighborhood of $0\in\bbbr^n$.}
$$ 
Let $\pi_k:\bbbr^n\to\bbbr^n$, $\pi_k(x_1,\ldots,x_n)=(x_1,\ldots,x_k,0,\ldots,0)$. Then $\Psi\circ f\circ\Phi^{-1}=\pi_k$, so
$f=\Psi^{-1}\circ\pi_k\circ\Phi$ in a neighborhood of $x$. If $\Phi_\eps$ and $(\Psi^{-1})_\eps$ are smooth approximations by mollification, then 
$f_\eps=(\Psi^{-1})_\eps\circ\pi_k\circ\Phi_\eps$ is $C^\infty$ smooth and it converges uniformly to $f$ in a neighborhood of $x$ as $\eps\to 0$. Clearly,
$\rank Df_\eps\leq k$ by the chain rule, since $\rank D\pi_k=k$.
\hfill $\Box$

\begin{remark}
It is easy to see that in fact $\rank f_\eps=k$ in a neighborhood of $x$, provided $\eps$ is sufficiently small. Indeed, $\Phi_\eps=\Phi*\varphi_\eps$ (approximation by mollification) so
$D\Phi_\eps=(D\Phi)*\varphi_\eps$. Since $\det(D\Phi(x))\neq 0$, for small $\eps>0$ we also have  that $\det(D\Phi_\eps)(x)\neq 0$ and hence $\Phi_\eps$ is a diffeomorphism near $x$. Similarly, $(\Psi^{-1})_\eps$ is a diffeomorphism near $0$.
\end{remark}

\section{Proof of Theorem~\ref{thm:main}}

In the first step of the proof we shall construct a mapping $F:\bbbb^{k+1}\to\bbbr^{m+1}$
defined on the unit ball $\bbbb^{k+1}=\bbbb^{k+1}(0,1)$,
with the properties announced by Theorem \ref{thm:main}.
\begin{lemma}
\label{lem:fromball}
Suppose that $m+1\leq k<2m-1$ and $\pi_k(\Sph^m)\neq 0$.
Then there exists a map $F\in C^1(\bbbb^{k+1},\bbbr^{m+1})$ with $\rank DF\leq m$ in $\bbbb^{k+1}$ and a Cantor set $E_F\subset \bbbb^{k+1}$ such that for every $x_o\in E_F$ and $1-|x_o|>\eps>0$ there is $\delta>0$ with the following property:\\ if $G\in C^{k-m+1}(\bbbb^{k+1}(x_o,\eps),\bbbr^{m+1})$ satisfies $|F(x)-G(x)|<\delta$ at all points $x\in \bbbb^{k+1}(x_o,\eps)$,\\ then $\rank DG\geq m+1$ on an open, non-empty set in $\bbbb^{k+1}(x_o,\eps)$.
\end{lemma}

Before we prove Lemma~\ref{lem:fromball}, let us show how Theorem~\ref{thm:main} follows from it.
To this end, let $\tilde{\bbbb}^{k+1}\subsetneq \bbbb^{k+1}$ be a ball concentric with $\bbbb^{k+1}$, containing the Cantor set $E_F$ and let $\Phi:\bbbb^{k+1}\to\bbbr^{k+1}$ be a diffeomorphism onto $\bbbr^{k+1}$ that is identity on $\tilde{\bbbb}^{k+1}$, so $F\circ\Phi^{-1}:\bbbr^{k+1}\to\bbbr^{m+1}$ coincides with $F$ on $\tilde{\bbbb}^{k+1}$ and hence in a neighborhood of the set $E_F$. Denote the points in $\bbbr^{\ell}$ and $\bbbr^r$ by
$$
(x,y)\in\bbbr^{k+1}\times\bbbr^{\ell-k-1}=\bbbr^\ell
\quad
\text{and}
\quad
(z,v)\in\bbbr^{m+1}\times\bbbr^{r-m-1}=\bbbr^{r}
$$
and let $\pi:\bbbr^r\to\bbbr^{m+1}$, $\pi(z,v)=z$ be the orthogonal projection.

It easily follows that the mapping
$$
\bbbr^{\ell}\ni (x,y)\longmapsto f(x,y):=(F\circ\Phi^{-1}(x),0)\in\bbbr^r
$$
satisfies the claim of Theorem~\ref{thm:main}
with $E=E_F\times\{ 0\}\subset\bbbr^{k+1}\times\bbbr^{\ell-k-1}=\bbbr^\ell$.

Indeed, in a neighborhood of $x_o\in E_F$, $f(x,y)=(F(x),0)$.

Suppose that $g\in C^{k-m+1}(\bbbb^\ell((x_o,0),\eps),\bbbr^r)$ is such that
$$
|f(x,y)-g(x,y)|<\delta
\quad
\text{for all $(x,y)\in\bbbb^\ell((x_o,0),\eps)$.}
$$
Then $G(x)=\pi(g(x,0))\in C^{k-m+1}(\bbbb^{k+1}(x_o,\eps),\bbbr^{m+1})$ satisfies
$$
|F(x)-G(x)|<\delta
\quad
\text{for all $x\in\bbbb^{k+1}(x_o,\eps)$}
$$
provided $\eps>0$ is so small that $f(x,y)=(F(x),0)$ for all $x\in\bbbb^{k+1}(x_o,\eps)$.

Hence $\rank DG\geq m+1$ on an open, non-empty set in $\bbbb^{k+1}(x_o,\eps)$ by Lemma~\ref{lem:fromball}.
Since $\rank Dg(x,0)\geq \rank DG(x)$ and the set $\{\rank Dg\geq m+1\}$ is open, $\rank Dg\geq m+1$ on an
open, non-empty subset of $\bbbb^\ell((x_o,0),\eps)$, which completes the proof of Theorem~\ref{thm:main}.
Therefore it remains to prove Lemma~\ref{lem:fromball}.

\begin{proof}[Proof of Lemma~\ref{lem:fromball}]
Let $\bbbi$ denote the unit cube $[-\frac{1}{2},\frac{1}{2}]^{m+1}$ in $\bbbr^{m+1}$.
Since, by assumption, $\pi_k (\Sph^m)\neq 0$ and $\partial\bbbi$ is homeomorphic to $\Sph^m$, there is a continuous mapping
$\hat{\phi}:\Sph^k\to\partial\bbbi$ that is not homotopic to a constant map.
Approximating $\hat{\phi}$ by standard mollification, we obtain a smooth mapping from $\Sph^k$ to $\R^{m+1}$, uniformly close to $\hat \phi$, with the image lying in a small neighborhood of $\partial \bbbi$.
Then, composing it with a $C^\infty$ smooth mapping $R$ that is homotopic to the identity and
maps a neighborhood of $\partial\bbbi$ onto $\partial\bbbi$ we obtain a mapping ${\phi}:\Sph^k\to\partial\bbbi$ that is not homotopic to a constant map and is
$C^\infty$ smooth as a mapping to $\bbbr^{m+1}$.

\noindent\begin{minipage}{0.67\textwidth}
\setlength{\parindent}{12pt}
 A smooth mapping $R:\bbbr^{m+1}\to\bbbr^{m+1}$ homotopic to the identity, that maps a neighborhood of $\partial\bbbi$ onto $\partial\bbbi$
can be defined by a formula
$$
R(x_1,x_2,\ldots,x_{m+1})=(\lambda_s(x_1),\lambda_s(x_2),\ldots,\lambda_s(x_{m+1})),
$$
where for $s\in (0,\frac{1}{4})$ the function $\lambda_s:\bbbr\to \bbbr$ is smooth, odd, non-decreasing and such that $\lambda_s(t)=t$ when $||t|-\frac{1}{2}|>2s$ and $\lambda(t)=1$ when  $||t|-\frac{1}{2}|<s$, see the graph on the right.
Taking $s\to 0$ gives a homotopy between $R$ and the identity.
\end{minipage}
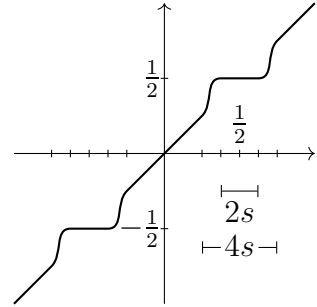
\begin{minipage}{0.29\textwidth}
\begin{tikzpicture}[scale=0.5]

\draw[->](0,0)--(8,0);
\draw[->](4,-4)--(4,4);
\foreach \i in {-3,-2.5,-2,-1.5,-1}
{
\draw (\i+4,0.1)--(\i+4,-0.1);
\draw (-\i+4,0.1)--(-\i+4,-0.1);
}
\draw[thick, shift={(4,0)}] (-4,-4) -- (-3,-3) to [out=45, in =180] (-2.5,-2) -- (-1.5,-2) to [out=0,in=-135] (-1,-1) -- (1,1) to [out=45,in=180] (1.5,2) -- (2.5,2) to [out=0,in=-135] (3,3)--(4,4);
\draw (3.9,-2)--(4.1,-2);
\draw (3.9,2)--(4.1,2);
\node[left] at (4.2,2) {$\frac{1}{2}$};
\node[left] at (4.2,-2) {$-\frac{1}{2}$};
\node[above] at (6,0) {$\frac{1}{2}$};
\draw[|-|] (5.5,-1)--(6.5,-1);
\node[below] at (6,-1) {$2s$};
\draw[|-] (5,-2.5)--(5.5,-2.5);
\draw[-|] (6.5,-2.5)--(7,-2.5);
\node[below] at (6,-1.9) {$4s$};
\node at (-0.5,0) {};
\end{tikzpicture}
\end{minipage}

Lemma~\ref{lem:fromball} is a simple consequence of the following result proved in
\cite[Lemma~5.1]{GH_Sard}.
(Note that in the statement of Lemma~5.1 in \cite{GH_Sard}, $k$ plays the role of $m$ and $m$ plays the role of $k$.)
The self-similarity property of the mapping $F$ in Lemma~\ref{GH} is explicitly stated in the proof of Lemma~5.1 in
\cite{GH_Sard}.
\begin{lemma}
\label{GH}
Suppose that $m+1\leq k<2m-1$ and $\pi_k(\Sph^m)\neq 0$.
Then there is a mapping $F\in C^1(\overbar{\bbbb}^{k+1},\bbbi)$
satisfying $\rank DF\leq m$ everywhere, such that $F$ maps the boundary $\partial \bbbb^{k+1}=\Sph^k$ to $\partial \bbbi$ and
$F|_{\partial\bbbb^{k+1}}=\phi$, where $\phi$ has been defined above.

Moreover, $F$ is self-similar in the following sense. There is a Cantor set $E_F\subset\bbbb^{k+1}$ such that for every $x_o\in E_F$
there is a sequence of balls $\bbbd_i\subset\bbbb^{k+1}$, $x_o\in \bbbd_i$, with radii convergent to zero, and similarity transformations
$$
\Sigma_i:\overbar{\bbbb}^{k+1}\to \overbar{\bbbd}_i,
\quad
\Sigma_i(\overbar{\bbbb}^{k+1})=\overbar{\bbbd}_i,
\qquad
T_i:\bbbr^{m+1}\to\bbbr^{m+1},
$$
each being a composition of a translation and scaling, such that
$$
T_i^{-1}\circ F|_{\overbar{\bbbd}_i}\circ\Sigma_i=F.
$$
\end{lemma}
Here the $C^1$ regularity of $F$ means that it is $C^1$ as a mapping into $\bbbr^{m+1}$, with the image being the cube $\bbbi$.

The mappings $T_i$ and $\Sigma_i$ are compositions  $T_i=\tau_{j_1}\circ\ldots\circ\tau_{j_i}$ and
$\Sigma_i=\sigma_{j_1}\circ\ldots\circ\sigma_{j_i}$ of similarity transformations $\tau_j$ and $\sigma_j$ that are used at the very end of the
proof of Lemma~5.1 in \cite{GH_Sard}. The Cantor set $E_F$ is the same as the Cantor set $C$ in the proof of Lemma~5.1 in \cite{GH_Sard}.

In other words, $F$ restricted to an arbitrarily small ball $\overbar{\bbbd}_i$ that contains $x_o$ is a scaled copy of
$F:\overbar{\bbbb}^{k+1}\to\bbbi$.

The mapping $F$ is obtained through an iterative construction, described in detail in \cite{GH_Sard}. We shall present here a sketch of that construction.
\begin{proof}[Sketch of the construction of the mapping $F$]\mbox{}\\

By assumption, $\pi_k(\Sph^m)\neq 0$. By Freudenthal's theorem (\cite[Corollary~4.24]{hatcher}), also $\pi_{k-1}(\Sph^{m-1})\neq 0$; 
let $h:\Sph^{k-1}\to\Sph^{m-1}$ be a mapping that is not homotopic to a constant.

We begin by choosing in the ball $\bbbb^{k+1}$ disjoint, closed balls $\bbbb_i$, $i=1,2,\ldots, N=n^{m+1}$, of radius $\frac{2}{n}$,
all inside $\frac{1}{2}\bbbb^{k+1}$. 
This is possible, if $n$ is chosen sufficiently large, since, for $n$ large, the $(k+1)$-dimensional volume of $\frac{1}{2}\bbbb^{k+1}$ is much 
larger than the sum of volumes of $\bbbb_i$, $2^{-(k+1)} \gg n^{m+1}{2}^{k+1}{n}^{-(k+1)}$.

We define a $C^\infty$-mapping $F$ in $\bbbb^{k+1}\setminus \bigcup_{i=1}^N \bbbb_i$; then, the same mapping is iterated inside each of the balls $\bbbb_i=\bbbb_{i,1}$, which defines $F$ outside a family of $N^2$ second generation balls $\bbbb_{i,2}$, and so on -- in this way we obtain a mapping which is $C^\infty$ outside a Cantor set. Finally, we extend $F$ continuously to the Cantor set $C$ defined by the subsequent generations of balls $\bbbb_{i,j}$, as the intersection $C=\bigcap_{j=1}^\infty \bigcup_{i=1}^{N^j}\bbbb_{i,j}$.

The mapping $F$ in $\bbbb^{k+1}\setminus \bigcup_{i=1}^N \bbbb_i$ is (in principle -- see comments below) defined as a composition of four steps (see Figure \ref{fig:1}):
\begin{enumerate}
\item First, we realign all the balls $\bbbb_i$ inside $\bbbb^{k+1}$, by a diffeomorphism $G_1$ equal to the identity near $\partial \bbbb^{k+1}$, so that the images of $\bbbb_i$ are identical, disjoint, closed balls lying along the vertical axis of $\bbbb^{k+1}$. Obviously, this diffeomorphism has to shrink the balls $\bbbb_i$ somewhat.
\item The next step, the mapping $H:\bbbb^{k+1}\to \bbbb^{m+1}$, is defined in the following way: it maps $(k-1)$-dimensional spheres centered at the vertical axis of $\bbbb^{k+1}$, lying in the hyperplane orthogonal to that axis, to $(m-1)$-dimensional spheres of the same radius, centered at analogous points on the vertical axis of $\bbbb^{m+1}$. On each such sphere, $H$ is an appropriately scaled copy of the mapping $h$. This way, $H$ restricted to any $k$-sphere centered on the axis 
(in particular to $\partial\bbbb_{k+1}$ and to $\partial (G_1(\bbbb_i))$) equals (up to scaling) to the suspension of $h$.

\item Next, we define the diffeomorphism $G_2$: we inflate the ball $\bbbb^{m+1}$ to $\frac{1}{2}\sqrt{m+1}\bbbb^{m+1}$, so that we can inscribe the unit cube $[-\frac{1}{2},\frac{1}{2}]^{m+1}$ in it, and inside that ball, we rearrange the $N$ balls $H(G_1(\bbbb_i))$, so that each of them is almost inscribed in one of the cubes of the grid obtained by partitioning the unit cube into $N=n^{m+1}$ cubes of edge length $\frac{1}{n}$.

\item Finally, we project $\frac{1}{2}\sqrt{m+1}\bbbb^{m+1} \setminus \bigcup_{i=1}^N G_2(H(G_1(\bbbb_i)))$ onto the $m$-dimensional skeleton of the grid: first, we project the outside of the unit cube onto the boundary of the cube using the nearest point projection $\pi$, then in each of the $N$ closed cubes of the grid we use the mapping $R$ defined in the proof of Lemma~\ref{lem:fromball}. Even though $\pi$ is not smooth, this composition turns out to be smooth (see \cite[Lemma 5.3]{GH_Sard}).
\end{enumerate}

In fact, this construction of $F$ outside $\bigcup_i \bbbb_i$ is almost correct -- the resulting mapping is not $C^\infty$, but Lipschitz: it is not differentiable at the points of the vertical axis, and some technical modifications are necessary to make it $C^\infty$. Similarly, some additional work is necessary to glue $F$ with scaled copies of $F$ in each of the balls $\bbbb_i$ in a differentiable way. These are purely technical difficulties, the details are provided in \cite{GH_Sard}.

The third iteration of that construction is depicted in Figure \ref{fig:2}.

One easily checks that the derivative of $F$ tends to $0$ as we approach the points of the Cantor set $C$, thus the limit mapping, extended to the whole $\bbbb^{k+1}$, is $C^1$.
For each point of $\bbbb^{k+1}\setminus C$, the image of its small neighborhood is mapped to the $m$-dimensional skeleton of the grid, thus $\rank DF\leq m$ at all these points, and since $DF=0$ at the points of $C$, the condition $\rank DF\leq m$ holds everywhere in $\bbbb^{k+1}$.

\begin{figure}
\begin{tikzpicture}[scale=0.7]
\begin{scope}[scale=0.6]
\node[scale=1.2] at (9.7,-4) {$\bbbb^{k+1}$};
\filldraw[fill=lightgray!20!white] (6,0) circle (4);
\draw[darkgray] (2,0) arc [x radius=4, y radius =2, start angle =-180, end angle=0];
\filldraw[draw=gray,fill=lightgray!40!white] (6,0) circle (2);

\filldraw[ball color=white] (5.2,0.8) circle (0.5);
\filldraw[fill=white,opacity=0.5] (5.2,0.8) circle (0.5);
\draw (5.2,0.8) circle (0.5);
\draw[darkgray] (4.7,0.8) arc [x radius=0.5, y radius =0.25, start angle =-180, end angle=0];
\filldraw[ball color=white] (6.8,0.8) circle (0.5);
\filldraw[fill=white,opacity=0.5] (6.8,0.8) circle (0.5);
\draw (6.8,0.8) circle (0.5);
\draw[darkgray] (6.3,0.8) arc [x radius=0.5, y radius =0.25, start angle =-180, end angle=0];
\filldraw[ball color=white] (5.2,-0.8) circle (0.5);
\filldraw[fill=white,opacity=0.5] (5.2,-0.8) circle (0.5);
\draw (5.2,-0.8) circle (0.5);

\draw[darkgray] (4.7,-0.8) arc [x radius=0.5, y radius =0.25, start angle =-180, end angle=0];
\filldraw[ball color=white] (6.8,-0.8) circle (0.5);
\filldraw[fill=white,opacity=0.5] (6.8,-0.8) circle (0.5);
\draw (6.8,-0.8) circle (0.5);
\draw[darkgray] (6.3,-0.8) arc [x radius=0.5, y radius =0.25, start angle =-180, end angle=0];

\draw[lightgray!50!gray] (4,0) arc [x radius=2, y radius =1, start angle =-180, end angle=0];
\end{scope}
\begin{scope}[shift={(8,0)},scale=0.6]
\node[scale=1.2] at (9.7,-4) {$\bbbb^{k+1}$};
\filldraw[fill=lightgray!20!white] (6,0) circle (4);

\draw[dashed,thick] (6,5)--(6,3.6);
\draw[dashed, gray] (6,3.6)--(6,-4);
\draw[dashed,thick] (6,-4)--(6,-5);
\filldraw (6,3.7) circle (0.05);
\filldraw[lightgray] (6,-3.8) circle (0.05);
\foreach \i in {0,1,2,3}
{
\begin{scope}[shift={(0,-1.5*\i)}]
\filldraw[ball color=white] (6,2.3) circle (0.5);
\filldraw[fill=white,opacity=0.5] (6,2.3) circle (0.5);
\draw (6,2.3) circle (0.5);
\draw[gray] (5.5,2.3) arc [x radius=0.5, y radius =0.25, start angle =-180, end angle=0];
\end{scope}
}

\draw[darkgray] (2,0) arc [x radius=4, y radius =2, start angle =-180, end angle=0];
\end{scope}
\draw[thick,->] (6.5,0)--(8.5,0);
\node[above] at (7.5,0) {$G_1$};


\begin{scope}[shift={(8,0)}]
\draw[thick,->] (6.5,0)--(8.5,0);
\node[above] at (7.5,0) {$H$};
\end{scope}

\begin{scope}[shift={(16,0)}, scale=0.6]
\node[scale=1.2] at (9.7,-4) {$\bbbb^{m+1}$};
\draw[dashed,thick] (6,5)--(6,3.6);
\draw[dashed, gray] (6,3.6)--(6,-4);
\draw[dashed,thick] (6,-4)--(6,-5);
\filldraw (6,3.7) circle (0.05);
\filldraw[lightgray] (6,-3.8) circle (0.05);
\filldraw[color= lightgray,opacity=0.4] (6,0) circle (4);
\draw[darkgray] (6,0) circle (4);
\foreach \i in {1,2,3,4}
{
\begin{scope}[shift={(0,-1.5*(\i-1))}]
\filldraw[draw=darkgray,fill=white!95!lightgray] (6,2.3) circle (0.5);
\end{scope}
}

\end{scope}



\draw[thick,->] (19,-3.5)--(18.2,-5.5);
\node at (18,-4.3) {$G_2$};

\begin{scope}[shift={(13,-9)}, scale=0.7]


\filldraw[color= gray,opacity=0.4] (6,0) circle (4);
\draw[darkgray] (6,0) circle (4);
\draw[thick,dashed] ({6+4*cos(45)},{4*sin(45)})--({6+4*cos(-45)},{4*sin(-45)})--({6+4*cos(-135)},{4*sin(-135)})--({6+4*cos(135)},{4*sin(135)})--cycle;
\draw[thick,dashed] (6,{4*sin(45)})--(6,{-4*sin(45)});
\draw[thick,dashed] ({6-4*sin(45)},0)--({6+4*sin(45)},0);
\foreach \i in {-1,1}
{
\foreach \j in {-1,1}
{
\begin{scope}[shift={({sin(45)*2*\i} ,{sin(45)*2*\j})}]
\filldraw[draw=darkgray,fill= white] (6,0) circle ({10/11*2*sin(45)});
\foreach \t in {-0.75,0.75}
{
\draw[gray,->] ({6+0.3*\t},{0.6})--({6+0.3*\t},1);
}

\foreach \t in {-0.75,0.75}
{
\draw[gray,->] ({6+0.3*\t},{-0.6})--({6+0.3*\t},-1);
}
\foreach \t in {-0.75,0.75}
{
\draw[gray,->] (5.4,{0.3*\t})--(5,{0.3*\t});
}

\foreach \t in {-0.75,0.75}
{
\draw[gray,->] (6.6,{0.3*\t})--(7,{0.3*\t});
}
\end{scope}
}
}
\foreach \t in {60,70,...,120}
{
\draw[gray,->] ({6+4*cos(\t)},{4*sin(\t)})--({6+4*cos(\t)},{3.6*sin(\t)});
}

\foreach \t in {-30,-20,...,30}
{
\draw[gray,->] ({6+4*cos(\t)},{4*sin(\t)})--({6+3.6*cos(\t)},{4*sin(\t)});
}
\foreach \t in {-60,-70,...,-120}
{
\draw[gray,->] ({6+4*cos(\t)},{4*sin(\t)})--({6+4*cos(\t)},{3.6*sin(\t)});
}
\foreach \t in {150,160,...,210}
{
\draw[gray,->] ({6+4*cos(\t)},{4*sin(\t)})--({6+3.6*cos(\t)},{4*sin(\t)});
}
\node[scale=1.2] at (10,-4.7) {$\frac{1}{2}\sqrt{m+1}\,\bbbb^{m+1}$};
\end{scope}
\begin{scope}[shift={(2,-9)},scale=0.7]

\filldraw[thick, fill=white] ({6+4*cos(45)},{4*sin(45)})--({6+4*cos(-45)},{4*sin(-45)})--({6+4*cos(-135)},{4*sin(-135)})--({6+4*cos(135)},{4*sin(135)})--cycle;
\draw[thick] (6,{4*sin(45)})--(6,{-4*sin(45)});
\draw[thick] ({6-4*sin(45)},0)--({6+4*sin(45)},0);
\end{scope}
\draw[thick,<-] (10,-9)--(13,-9);
\node[above] at (11.5,-9) {$R$};

\end{tikzpicture}
\caption{The construction of $F$ in $\bbbb^{m+1}\setminus \bigcup_{i=1}^N \bbbb_i$.}\label{fig:1}
\end{figure}
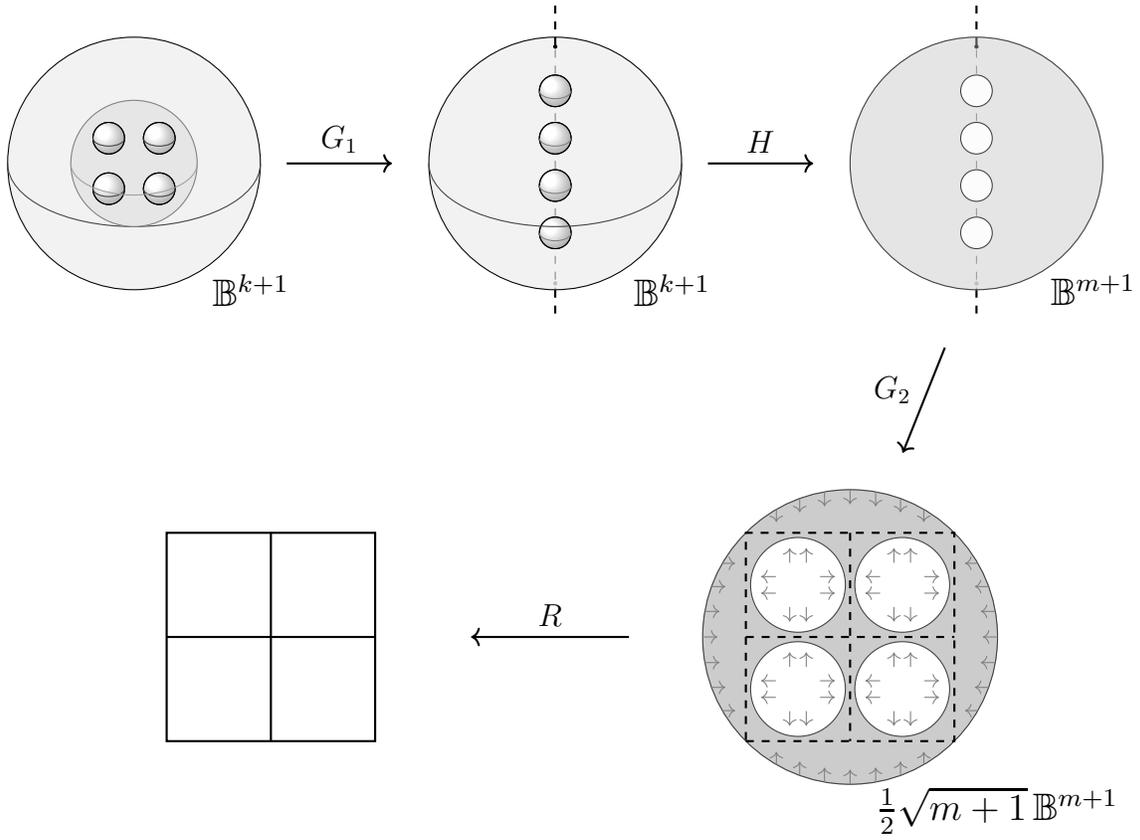
\begin{figure}
\begin{tikzpicture}[scale=0.8]
\filldraw[thick,fill= lightgray!10!white] (6,0) circle (4);
\draw[darkgray] (2,0) arc [x radius=4, y radius=2, start angle=-180, end angle=0];
\draw[darkgray, dotted] (2,0) arc [x radius=4, y radius=2, start angle=180, end angle=0];
\filldraw[fill= lightgray!20!white,draw=gray] (6,0) circle (2);

\begin{scope}[shift={(4.35,0.6)},scale=0.14]

\filldraw[draw=lightgray!50!white, fill=lightgray!40!white] (6,0) circle (2);
\draw[lightgray!40!gray] (4,0) arc [x radius=2, y radius =1, start angle =-180, end angle=0];
\filldraw[draw=lightgray!40!white,fill=lightgray!30!white] (6,0) circle (4);
\draw (6,0) circle (4);
\draw[gray] (2,0) arc [x radius=4, y radius =2, start angle =-180, end angle=0];
\filldraw[thick,fill=lightgray!30!white] (5.1,0.9) circle (0.7);
\foreach \i in {-1,1}
{
\foreach \j in {-1,1}
{\filldraw (5.1+0.15*\i,0.9+0.15*\j) circle (0.07);}
}
\draw[gray] (4.7,0.8) arc [x radius=0.5, y radius =0.25, start angle =-180, end angle=0];
\filldraw[thick,fill=lightgray!30!white] (6.9,0.9) circle (0.7);
\foreach \i in {-1,1}
{
\foreach \j in {-1,1}
{\filldraw (6.9+0.15*\i,0.9+0.15*\j) circle (0.07);}
}
\draw[gray] (6.3,0.8) arc [x radius=0.5, y radius =0.25, start angle =-180, end angle=0];
\filldraw[thick,fill=lightgray!30!white] (5.1,-0.9) circle (0.7);
\foreach \i in {-1,1}
{
\foreach \j in {-1,1}
{\filldraw (5.1+0.15*\i,-0.9+0.15*\j) circle (0.07);}
}
\draw[gray] (4.7,-0.8) arc [x radius=0.5, y radius =0.25, start angle =-180, end angle=0];
\filldraw[thick,fill=lightgray!30!white] (6.9,-0.9) circle (0.7);
\foreach \i in {-1,1}
{
\foreach \j in {-1,1}
{\filldraw (6.9+0.15*\i,-0.9+0.15*\j) circle (0.07);}
}
\draw[gray] (6.3,-0.8) arc [x radius=0.5, y radius =0.25, start angle =-180, end angle=0];
\end{scope}
\begin{scope}[shift={(6,0.6)},scale=0.14]

\filldraw[draw=lightgray!50!white, fill=lightgray!40!white] (6,0) circle (2);
\draw[lightgray!40!gray] (4,0) arc [x radius=2, y radius =1, start angle =-180, end angle=0];
\filldraw[draw=lightgray!40!white,fill=lightgray!30!white] (6,0) circle (4);
\draw (6,0) circle (4);
\draw[gray] (2,0) arc [x radius=4, y radius =2, start angle =-180, end angle=0];
\filldraw[thick,fill=lightgray!30!white] (5.1,0.9) circle (0.7);
\foreach \i in {-1,1}
{
\foreach \j in {-1,1}
{\filldraw (5.1+0.15*\i,0.9+0.15*\j) circle (0.07);}
}
\draw[gray] (4.7,0.8) arc [x radius=0.5, y radius =0.25, start angle =-180, end angle=0];
\filldraw[thick,fill=lightgray!30!white] (6.9,0.9) circle (0.7);
\foreach \i in {-1,1}
{
\foreach \j in {-1,1}
{\filldraw (6.9+0.15*\i,0.9+0.15*\j) circle (0.07);}
}
\draw[gray] (6.3,0.8) arc [x radius=0.5, y radius =0.25, start angle =-180, end angle=0];
\filldraw[thick,fill=lightgray!30!white] (5.1,-0.9) circle (0.7);
\foreach \i in {-1,1}
{
\foreach \j in {-1,1}
{\filldraw (5.1+0.15*\i,-0.9+0.15*\j) circle (0.07);}
}
\draw[gray] (4.7,-0.8) arc [x radius=0.5, y radius =0.25, start angle =-180, end angle=0];
\filldraw[thick,fill=lightgray!30!white] (6.9,-0.9) circle (0.7);
\foreach \i in {-1,1}
{
\foreach \j in {-1,1}
{\filldraw (6.9+0.15*\i,-0.9+0.15*\j) circle (0.07);}
}
\draw[gray] (6.3,-0.8) arc [x radius=0.5, y radius =0.25, start angle =-180, end angle=0];
\end{scope}

\begin{scope}[shift={(4.35,-0.8)},scale=0.14]

\filldraw[draw=lightgray!50!white, fill=lightgray!40!white] (6,0) circle (2);
\draw[lightgray!40!gray] (4,0) arc [x radius=2, y radius =1, start angle =-180, end angle=0];
\filldraw[draw=lightgray!40!white,fill=lightgray!30!white] (6,0) circle (4);
\draw (6,0) circle (4);
\draw[gray] (2,0) arc [x radius=4, y radius =2, start angle =-180, end angle=0];
\filldraw[thick,fill=lightgray!30!white] (5.1,0.9) circle (0.7);
\foreach \i in {-1,1}
{
\foreach \j in {-1,1}
{\filldraw (5.1+0.15*\i,0.9+0.15*\j) circle (0.07);}
}
\draw[gray] (4.7,0.8) arc [x radius=0.5, y radius =0.25, start angle =-180, end angle=0];
\filldraw[thick,fill=lightgray!30!white] (6.9,0.9) circle (0.7);
\foreach \i in {-1,1}
{
\foreach \j in {-1,1}
{\filldraw (6.9+0.15*\i,0.9+0.15*\j) circle (0.07);}
}
\draw[gray] (6.3,0.8) arc [x radius=0.5, y radius =0.25, start angle =-180, end angle=0];
\filldraw[thick,fill=lightgray!30!white] (5.1,-0.9) circle (0.7);
\foreach \i in {-1,1}
{
\foreach \j in {-1,1}
{\filldraw (5.1+0.15*\i,-0.9+0.15*\j) circle (0.07);}
}
\draw[gray] (4.7,-0.8) arc [x radius=0.5, y radius =0.25, start angle =-180, end angle=0];
\filldraw[thick,fill=lightgray!30!white] (6.9,-0.9) circle (0.7);
\foreach \i in {-1,1}
{
\foreach \j in {-1,1}
{\filldraw (6.9+0.15*\i,-0.9+0.15*\j) circle (0.07);}
}
\draw[gray] (6.3,-0.8) arc [x radius=0.5, y radius =0.25, start angle =-180, end angle=0];
\end{scope}
\begin{scope}[shift={(6,-0.8)},scale=0.14]

\filldraw[draw=lightgray!50!white, fill=lightgray!40!white] (6,0) circle (2);
\draw[lightgray!40!gray] (4,0) arc [x radius=2, y radius =1, start angle =-180, end angle=0];
\filldraw[draw=lightgray!40!white,fill=lightgray!30!white] (6,0) circle (4);
\draw (6,0) circle (4);
\draw[gray] (2,0) arc [x radius=4, y radius =2, start angle =-180, end angle=0];
\filldraw[thick,fill=lightgray!30!white] (5.1,0.9) circle (0.7);
\foreach \i in {-1,1}
{
\foreach \j in {-1,1}
{\filldraw (5.1+0.15*\i,0.9+0.15*\j) circle (0.07);}
}
\draw[gray] (4.7,0.8) arc [x radius=0.5, y radius =0.25, start angle =-180, end angle=0];
\filldraw[thick,fill=lightgray!30!white] (6.9,0.9) circle (0.7);
\foreach \i in {-1,1}
{
\foreach \j in {-1,1}
{\filldraw (6.9+0.15*\i,0.9+0.15*\j) circle (0.07);}
}
\draw[gray] (6.3,0.8) arc [x radius=0.5, y radius =0.25, start angle =-180, end angle=0];
\filldraw[thick,fill=lightgray!30!white] (5.1,-0.9) circle (0.7);
\foreach \i in {-1,1}
{
\foreach \j in {-1,1}
{\filldraw (5.1+0.15*\i,-0.9+0.15*\j) circle (0.07);}
}
\draw[gray] (4.7,-0.8) arc [x radius=0.5, y radius =0.25, start angle =-180, end angle=0];
\filldraw[thick,fill=lightgray!30!white] (6.9,-0.9) circle (0.7);
\foreach \i in {-1,1}
{
\foreach \j in {-1,1}
{\filldraw (6.9+0.15*\i,-0.9+0.15*\j) circle (0.07);}
}
\draw[gray] (6.3,-0.8) arc [x radius=0.5, y radius =0.25, start angle =-180, end angle=0];
\end{scope}

\draw[lightgray!30!gray,draw=gray] (4,0) arc [x radius=2, y radius =1, start angle =-180, end angle=0];

\draw[help lines] (6.89,0.81) -- (10.18,5.1);
\draw[help lines] (7.07,0.68) -- (11.65,3.72);

\begin{scope}[shift={(9.5,4.5)}, scale=0.25]
\filldraw[thick,fill= lightgray!10!white] (6,0) circle (4);
\draw[gray] (2,0) arc [x radius=4, y radius=2, start angle=-180, end angle=0];
\draw[darkgray, dotted] (2,0) arc [x radius=4, y radius=2, start angle=180, end angle=0];
\filldraw[fill= lightgray!20!white,draw=lightgray] (6,0) circle (2);

\begin{scope}[shift={(4.35,0.6)},scale=0.14]
\draw[thick] (6,0) circle (4);
\draw[gray] (2,0) arc [x radius=4, y radius =2, start angle =-180, end angle=0];
\end{scope}
\begin{scope}[shift={(6,0.6)},scale=0.14]
\draw[thick] (6,0) circle (4);
\draw[gray] (2,0) arc [x radius=4, y radius =2, start angle =-180, end angle=0];
\end{scope}
\begin{scope}[shift={(4.35,-0.8)},scale=0.14]
\draw[thick] (6,0) circle (4);
\draw[gray] (2,0) arc [x radius=4, y radius =2, start angle =-180, end angle=0];
\end{scope}
\begin{scope}[shift={(6,-0.8)},scale=0.14]
\draw[thick] (6,0) circle (4);
\draw[gray] (2,0) arc [x radius=4, y radius =2, start angle =-180, end angle=0];
\end{scope}
\end{scope}


\begin{scope}[scale=1.2]
\foreach \i in {0,1,...,8}
{
\draw[thick] (13+0.5*\i,2)--(13+0.5*\i,-2);
\draw[thick] (13,2-0.5*\i)--(17,2-0.5*\i);
}
\end{scope}

\draw[->] (12,0)--(14,0);
\node[scale=1.3] at (13,0.8) {$F$};
\end{tikzpicture}
\caption{The third iteration: $F$ outside the third generation of balls $\bigcup_i \bbbb_{3,i}$.}\label{fig:2}
\end{figure}
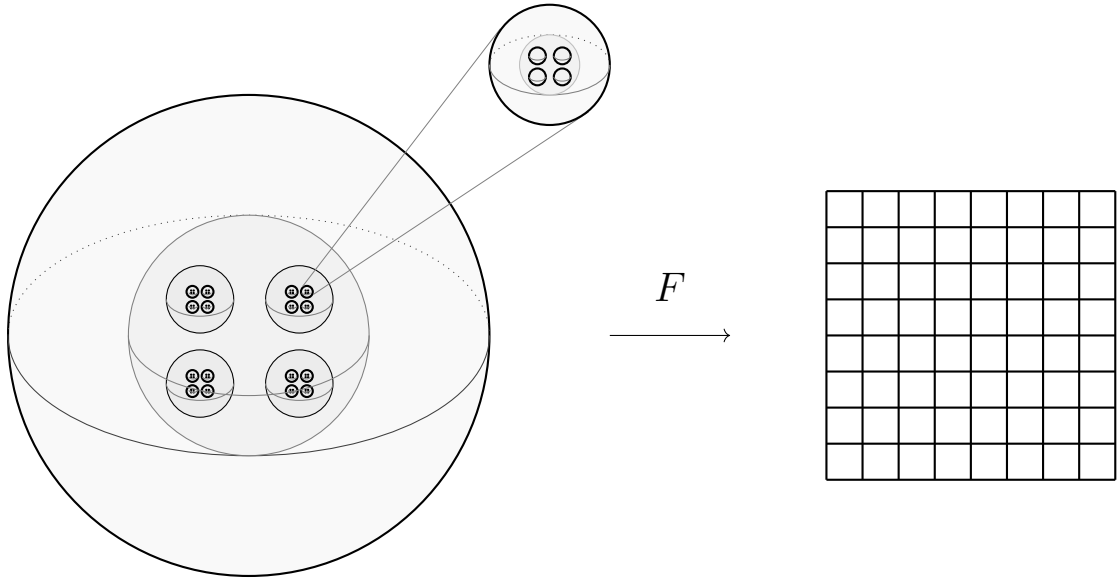

\end{proof}

Lemma~\ref{GH} allows us to complete the proof of Lemma~\ref{lem:fromball} as follows.
Let $x_o\in E_F$ and $1-|x_o|>\eps>0$ be given.
Suppose to the contrary, that there is a sequence
$G_j\in C^{k-m+1}(\bbbb^{k+1}(x_o,\eps),\bbbr^{m+1})$ with $\rank DG_j\leq m$, that is uniformly convergent
to $F$ on $\bbbb^{k+1}(x_o,\eps)$.

Let $\bbbd_i$ be a sequence of balls convergent to $x_o$ as in the statement of Lemma~\ref{GH}.
If $i$ is sufficiently large, then $\overbar{\bbbd}_i\subset\bbbb^{k+1}(x_o,\eps)$ and the sequence $G_j$ converges uniformly to $F$ on
$\overbar{\bbbd}_i$. Hence
$$
\tilde{G}_j:=T_i^{-1}\circ G_j|_{\overbar{\bbbd}_i}\circ\Sigma_i:\overbar{\bbbb}^{k+1}\to\bbbr^{m+1}
$$
converges uniformly to
$$
T_i^{-1}\circ F|_{\overbar{\bbbd}_i}\circ\Sigma_i=F:\overbar{\bbbb}^{k+1}\to\bbbi.
$$
Obviously, $\rank D\tilde{G}_j\leq m$.
Since $\tilde{G}_j$ is uniformly close to $F$ on $\partial{\bbbb}^{k+1}$ and
$F|_{\partial\bbbb^{k+1}}:\Sph^k\to\partial\bbbi$ is not homotopic to a constant map, it easily follows that for $j$ sufficiently large
the image $\tilde{G}_j({\bbbb}^{k+1})$ contains the cube $\frac{1}{2}\bbbi$ that is concentric with $\bbbi$ and has half the diameter
(as otherwise, using a projection onto the boundary of the cube, one could construct a homotopy of
$F|_{\partial\bbbb^{k+1}}:\Sph^k\to\partial\bbbi$ to a constant map).

Recall that according to Sard's theorem \cite{sard,sternberg}, the map $\tilde{G}_j\in C^{k-m+1}$ maps the set of its critical points to a set of measure zero.
Since $\rank D\tilde{G}_j\leq m$, all points in $\bbbb^{k+1}$ are critical, so the set $\tilde{G}_j({\bbbb}^{k+1})$ has measure zero, which contradicts the fact
that it contains the cube $\frac{1}{2}\bbbi$. The proof is complete.
\end{proof}

\end{document}